\documentclass[11pt]{article}
\usepackage{graphicx} % Required for inserting images
\usepackage{amsmath, amsthm, amssymb, amsfonts}
\usepackage{thmtools}
\usepackage{graphicx}
\usepackage{setspace}
\usepackage[text={6.5in, 9in}, centering]{geometry}
\usepackage{float}
\usepackage{hyperref}
\usepackage[utf8]{inputenc}
\usepackage[english]{babel}
\usepackage{framed}
\usepackage[dvipsnames]{xcolor}
\usepackage{tcolorbox}
\usepackage{physics}
\usepackage{csquotes}

\colorlet{LightGray}{White!90!Periwinkle}
\colorlet{LightOrange}{Orange!15}
\colorlet{LightGreen}{Green!15}
\colorlet{LightYellow}{Yellow!15}
\colorlet{LightBlue}{Blue!15}

\declaretheoremstyle[name=Theorem,]{thmsty}
\declaretheorem[style=thmsty,numberwithin=section]{theorem}

\declaretheoremstyle[name=Proposition,]{prosty}
\declaretheorem[style=prosty,numberlike=theorem]{proposition}

\declaretheoremstyle[name=Corollary,]{prosty}

\declaretheoremstyle[name=Lemma,]{prosty}
\declaretheorem[style=prosty,numberlike=theorem]{lemma}

\declaretheoremstyle[name=Definition,]{prosty}

\newcommand{\R}{\mathbb{R}}

\newcommand{\Z}{\mathbb{Z}}

\definecolor{darkblue}{rgb}{0,0,0.7}

\newcommand{\crr}[1]{{\color[rgb]{1,0,0} #1}} % red
\newcommand{\al}{\alpha}

\newcommand{\de}{\delta}
\newcommand{\e}{\epsilon}
\newcommand{\ga}{{\gamma}}

\newcommand{\la}{\lambda}
\newcommand{\La}{\Lambda}

\newcommand{\si}{\sigma}

\renewcommand{\th}{\theta}
\newcommand{\I}{\infty}

\title{Regularity, uniqueness and the relative size of small and large scales in SQG flows}
\author{Z. Akridge and Z. Bradshaw} 

\begin{document}
\maketitle

\begin{abstract}  
The problem of regularity and uniqueness are open for the supercritically dissipative surface quasi-geostrophic equations in certain classes. In this note we examine the extent to which small or large scales are necessarily active both for the temperature in a hypothetical blow-up scenario and for the error in hypothetical non-uniqueness scenarios, the latter understood within the class of Marchand's solutions. This extends prior work for the 3D Navier-Stokes equations. The extension is complicated by the fact that mild solution techniques are unavailable for supercritical SQG. This forces us to develop a new approach using energy methods and Littlewood-Paley theory. 
\end{abstract}

\section{Introduction}

We define the $\al$-dissipative surface quasi-geostrophic  equation to be 
\begin{equation}\label{eq.SQG}\tag{SQG$_\al$}
    \partial_t\theta + u\cdot \nabla \theta + (-\Delta)^{\alpha}\theta = 0;
    \quad u = (-\mathcal{R}_{x_2}\theta,\mathcal{R}_{x_1}\theta),
\end{equation}
where $\theta:\R\cross \R^2\to \R$ and $\mathcal{R}_{x_n}$ is the Riesz transform with respect to the variable $x_n$ and $\al \in [0,1]$ is fixed. In certain instances, these equations have physical relevance \cite{held1995surface,lapeyre2017surface}. The are  more generally useful as a toy model for regularity questions for the Euler and Navier-Stokes equations. 

When $\alpha\in (1/2,1]$ the problem is sub-critical and global regularity is known. When $\alpha=1/2$ the problem is critical, in which case global regularity is  also known, but is considerably harder to achieve. 
A selection of results addressing these cases are \cite{ConstantinCordobaGancedoWu,ConstantinMajdaTabak,ConstTarfVicolCMP,VariationsThemeCaffVass2009,KiselevNazarovVolbergInventiones2007,ConstWu,ConstVicolGAFA}.
When $0<\alpha<1/2$ global regularity is, in general, an open problem. The same can be said of uniqueness of weak solutions, although certain classes of \textit{very} weak solutions are known to exhibit non-uniqueness \cite{BSV}, see also \cite{chengkwon2020non,novackquasigeostrophic}.

In this paper we consider $\alpha\in (0,1/2)$, which is the supercritical range for the dissipative problem. It is supercritical in the sense that quantities which have a maximum principle are not known to imply global regularity or well-posedness. Our present goal is to analyze sufficient conditions for regularity and uniqueness  which are based on the relative sizes of small and large scales, interpreted through a Littlewood-Paley lens. This analysis is motivated by past work for the  3D Navier-Stokes equations by Albritton and Bradshaw \cite{AB2} and Bradshaw \cite{B1}. The technical details are, however, substantially different. In particular, both \cite{AB2} and \cite{B1} primarily use mild solution estimates. These are known to break down in the supercritical SQG regime and we therefore need to develop alternative estimates based on energy methods. In addition to shedding new light on the dynamics of hypothetical SQG blow-up and non-uniqueness, the results of this paper  demonstrate that the ideas in \cite{AB2,B1}  can be adapted to PDEs with different qualitative properties compared to the 3D Navier-Stokes equations.

\subsection{Regularity criteria}

SQG has a natural scaling: If $\th$ is a solution to \eqref{eq.SQG}, then so is $\th_\la := \la^{2\al-1}\th(\la x,\la^{2\al}t)$. We therefore see that the Sobolev space $\dot H^{2-2\al}$ is a scaling invariant space. We will refer to such spaces as \textit{critical}. Roughly speaking, local well-posedness is known in $\dot H^s$ for $s>2-2\al$ with partial results if $s=2-2\al$, see \cite{Vincent} for a comprehensive review of results in this direction, while global existence of weak solutions, in the sense of Marchand, is known for $-4/3<s<2-2\al$ \cite{marchand}.

Our first regularity result should be viewed through the lens of subcritical local well-posedness theories. 
Let us recall Ju's local well-posedness theorem in Sobolev spaces \cite{Ju}: 
If $s>2-\alpha$ where $\alpha\in (0,1/2)$ and $\theta_0\in \dot H^{s}\cap L^2$, then there exists $T=T(\|u_0\|_{\dot H^s})>0$ and a unique strong solution $\theta$ to \eqref{eq.SQG} so that 
\[
\th\in L^\I(0,T;H^s )\cap L^2(0,T;H^{s+\al}),
\]
with
\[
\| \th (t)\|_{\dot H^{s}} \leq 2\| \th_0\|_{\dot H^{s}},
\]
provided $t\in (0,T_{min})$ where
\[
T_{min} \| \th_0\|_{\dot H^s}^{\frac {2\al} {s-2+2\al}} = C_0(s,\al),
\]
and $C_0(s,\al)$ is a universal constant depending on $s$ and $\al$ and $T_{min} < T$. To confirm this in \cite{Ju}, refer to the inequality below \cite[(3.5)]{Ju}. We will view $T_{min}$ as the \textit{minimum guaranteed time of existence}.  

 Our first regularity result says that, if $\|\th_0\|_{\dot H^s}$ is held fixed but activity is concentrated on high modes, then Ju's existence time can be extended. 

\begin{theorem}[Local well-posedness up to a prescribed time]\label{thrm.reg1}
    Let $s\in (2-2\al , 2- \al )$ and assume $\th_0\in L^2\cap \dot H^s$. 
    Let $T_{*}>T_{min}$ be given.  
    Then, there exists a constant $C=C(s,\al)$ so that, for $J=J(t)$   satisfying  
 \[
2^{2\al J} \geq C \|\th_0\|_{\dot H^s}^{1 + \frac {4\al} {s-2+2\al}} T_*^{1+\frac {s-2+s\al}{2\al}},
\]
 if
  \begin{align} \label{freqSp}
  \frac {\|  \Delta_{<J} \th_0 \|_{\dot H^s}} {\| \Delta_{\geq J} \th_0\|_{\dot H^s}} \leq \frac \ga {4 C_b},
\end{align}
    where $C_b$ is a universal constant and 
    \[
     \ga = \bigg(  \frac {T_{min}} {T_*} \bigg)^{\frac {s-2+2\al} {2\al}}       ,
    \]
    then $\th$ can be continued smoothly up to time $T_*$ and satisfies 
    \[
    \sup_{0<t<T_*}\| \th (t)\|_{\dot H^s}\leq 2 \|  \th_0\|_{\dot H^s}. 
    \] 
\end{theorem}

%The assumption that $\th_0\in L^2$ is not used in the proof and is merely included to ensure the full strength of Ju's local well-posedness result applies.  

The idea behind Theorem \ref{thrm.reg1} is that a solution $\th$ of \eqref{eq.SQG} can be viewed as a linear part---this would solve the fractional heat equation---and a nonlinear part which appears with a time integral as in Duhamel's principle. The fractional heat part wants to dissipate activity at small scales. So, if the bulk of the activity is concentrated at small scales, then the linear part of the solution will decay rapidly. The nonlinear part starts at zero and grows at an algebraic rate. Taken together and assuming activity is concentrated at high enough frequencies, these facts suggest that we can find a time at which the $\dot H^s$ norm has become as small as we want to guarantee that, upon re-solving, the solution extends past $T_*$. This idea is visible in Proposition \ref{PROP}, which provides the heavy lifting for our regularity results.   Let us note that while this idea is motivated by work for 3D Navier-Stokes \cite{AB2}, its implementation here is quite different due to the fact that mild solution estimates aren't useful in the supercritical SQG range---indeed the preceding intuitive picture is slightly misleading as the Duhamel part cannot be controlled directly. This necessitated re-imagining the argument from \cite{AB2} in terms of energy inequalities, which is the main technical contribution of this paper. In \cite{AB2}, the solutions were mild so their linear and non-linear parts are built into the solution formula. It is less clear when using energy estimates how to exploit the fact that small scales are eliminated rapidly by diffusion. We are able to identify the correct mechanism by considering differential inequalities for the energy associated with individual Littlewood-Paley modes---see the proof of  Proposition \ref{PROP}.

Our next theorem considers a Type 1 blow-up scenario and establishes an upper bound on the rate at which small scales can activate. 

\begin{theorem}[Blow-up criteria]  \label{thrm.reg2}    Let $s\in (2-2\al,2-\al)$. Let $\th :\R^2\times (0,T_{max}) \to \R$ belong to $L^\I_{loc} ((0,T_{max}); \dot H^{s})\cap L^\I(0,T_{max};L^2)$ be a unique strong solution on $\R^2\times (0,T_{max})$ which is singular at $T_{max}$. Assume further that there exists $C_*$ so that
\[  \bigg(  \frac {C_0(s,\al)} {T_{max}-t}    \bigg)^{(s-2+2\al) / (2\al)} \leq \|\th(t)\|_{\dot H^s}\leq  \bigg(  \frac {C_*} {T_{max}-t}    \bigg)^{(s-2+2\al) / (2\al)}.\]
Then, there exists $J$ with $2^{2\al J}\sim (T_{max}-t)^{-1}$ and a constant $c_*>0$  so that 
\[
\inf_{0<t<T_{max}} \frac {\|  \Delta_{<J} \th (t)\|_{\dot H^s}} {\|  \Delta_{\geq J} \th (t)\|_{\dot H^s}} >c_*. 
\]
\end{theorem}

The lower bound stated in the theorem does not need to built into the assumptions as it is a necessary condition for blow-up.

This theorem reveals that low modes (as identified by a time dependent threshold) must be active to some extent as a hypothetical singularity develops. This is somewhat counterintuitive as high modes are known to drive singularity formation. In essence, the Type 1 condition restricts the rate at which activity can pass to very small scales, indicating complimentary scales must be active. 
There is a vast literature on conditional regularity criteria for the supercritical SQG equations, some of which is reviewed before Theorem \ref{thrm.reg3}. We feel it is presently appropriate to only mention thematically similar papers of Dai \cite{dai} (see also \cite{CS2}) and Wu \cite{Wu} which explore the role of high frequencies on one hand and sparse sets in the Fourier variable on the other.

One motivation for   Albritton and Bradshaw's work on Navier-Stokes \cite{AB2}, on which the present note is based, concerns a regularity program of Gruji\'c and coauthors \cite{Gru,BrFaGr,GrXu1,GrXu2}. We believe this program can be re-formulated through the lens of ``frequency sparseness,'' which amounts to conditions like \eqref{freqSp}.  This approach avoids the use of analyticity coupled with  $L^\I$ estimates. For the SQG problem, this avoidance seems necessary because $L^\I$ is not a sub-critical space and, consequently, the analyticity argument in \cite{Gru}, which involve the maximum modulus principle, cannot be applied directly.  Some of the later results of Gruji\'c and Xu \cite{GrXu2} relate to the generalized Navier-Stokes with dissipation below the Lions exponent $(-\Delta)^{5/4}$. We believe that the present formulation of ``frequency sparseness'' could allow these results to be adapted to the supercritical SQG equations and are curious to explore this in future work.

We conclude our discussion of regularity with a new proof of an endpoint regularity criteria. This should be compared to \cite[Theorem 3.4]{Dong1} which involves regularity criteria where the time-integral is $L^{r_0}$ and $r_0<\I$; see also \cite{ConstWu}. When $r_0=\I$, there is a corresponding statement due to Chae and Lee in $B^{2-2\al}_{2,1}$. Chae also formulated similar results in terms of $\nabla^\perp \theta$ \cite{Chae}, see also Yuan's paper \cite{Yuan}. An endpoint case is considered by Dong and Pavlovi\'c \cite{Dong2} (see also \cite{CS1}) assuming $\th \in C (0,T; \dot B^{1-2\al}_{\I,\I})$, which is close to  a result involving the time $L^\I$-norm and, in light of the version of their result which allows small jumps, can be used to prove a smallness-implies-regularity version. Because $\dot B^{2-2\al}_{2,\I}$ embeds continuously in $\dot B^{1-2\al}_{\I,\I}$, the result of Dong and Pavlovi\'c encompasses some of what follows. A novel aspect of our result compared to these others is that only scales larger than a function of $\| \th \|_{\dot H^s}(t)$ appear in the regularity criteria---a similar theme appears in \cite{BG} for the Navier-Stokes equations.
Based on this difference, as well as the fact that the approach of the proof is new compared to \cite{Dong2}, we believe it is worth including. 

\begin{theorem} \label{thrm.reg3}
    Let $\theta_0\in H^{2-2\alpha + \epsilon=:s}$ where $0<\epsilon<\alpha$. 
     Assume $\th\in C((0,T_*);\dot H^s )$ is a strong solution to \eqref{eq.SQG} with data $\th_0$ and $T_*$ is the first possible blow-up time.  Then there exists a small constant $\epsilon_*>0$ and function $J(t)$ with  
     \[ 
2^{(2-s+2\al)J(t)} = C(s,\al) \|\th(t)\|_{\dot H^s},
\] 
     so that if
    \[\sup_{0<t<T^*}\|\Delta_{<J(t)} \theta(x,t)\|_{\dot B^{2-2\alpha}_{2,\infty}}\leq \epsilon_*,\]
    then $T^*$ is not a blow-up time and the solution can be smoothly extended past  $T_*$.
\end{theorem}

\subsection{Uniqueness criteria}

Non-uniqueness can, in theory, ruin the predictive power of simulations. It is therefore important to understand the severity of a hypothetical non-uniqueness scenario, which can be analyzed by providing growth rates for the error, which is the difference between two solutions, or establishing necessary properties of the error. The latter could be obtained by formulating uniqueness criteria in terms of the error.  It is not obvious what such uniqueness criteria would look like. Some work in this direction has been carried out for the 3D Navier-Stokes equations by the second author \cite{B1}.  Presently, the most likely candidates for non-uniqueness  for 3D Navier-Stokes within a physical class are based on self-similar solutions \cite{GuSv,JS1,JS2,AbC}. The scaling of these solutions implies certain relationships between the scales of the error. Denoting   by $w$ the difference of two self-similar solutions sharing the self-similar initial data $u_0$, we have, for instance, that individual Littlewood-Paley modes vanish as $t\to 0^+$, in particular, \[
\|  \Delta_{< J} w\|_{L^\infty}(t)\lesssim_{{u_0}} 2^{4J} t^{3/2},
\]
provided the two solutions are in a reasonable class \cite{B1}. On the other hand, due to the exact scaling of the error, 
\[
\| w \|_{L^\infty}(t)= t^{-1/2}\|w\|_{L^\infty}(1).
\]
So, as $t\to 0^+$, $\| w_{\geq J}\|_{L^\infty}$ is blowing up while $\|w_{<J}\|_{L^\infty}\to 0$, indicating activity is concentrating on smaller and smaller scales. In other words, in self-similar non-uniqueness for the Navier-Stokes equations the activity in the error would start at infinitesimally small scales and propagate to larger scales as time passes. One goal of \cite{B1} is to extend this to general classes which do not have an exact scaling property. Presently, we pursue this goal for the supercritical surface quasi-geostrophic equations. Note that hypothetical self-similar non-uniqueness has not been thoroughly explored for SQG, although the existence of large self-similar solutions has been developed in the critical case by Albritton and the second author as a preliminary step in this direction \cite{AB2}. Such solutions have not been constructed in the supercritical case. 

We presently adapt some of the ideas of \cite{B1} to supercritical SQG. 
Compared to \cite{B1}, it is again necessary to formulate proofs in terms of energy methods.  We have in mind scenarios where a rough initial data possibly generates multiple solutions which are regular at $t>0$. This is the case for the hypothetical self-similar solutions to the Navier-Stokes equations on which our analysis is inspired. For this reason, our solutions are viewed as possibly non-unique but are also strong solutions for $t>0$---this manifests, e.g., as the assumption that $\|\nabla \theta\|_\I(t)\in L^\I_{loc}( (0,T))$. We stress that they do not belong to a known uniqueness class on the full interval $[0,T]$.

Our first theorem asserts that high modes in the error cannot be subordinate to low modes as $t\to 0^+$, assuming the error is non-zero.
 
\begin{theorem}\label{thrm.uniqueness1}
    Let $\theta_1,\theta_2$ be weak solutions of SQG in the sense of Marchand with the same initial data $\theta_0\in L^q\cap L^2$ for some $q\in (1,\infty)$. Let $p=\frac{2q}{q-1}\in(2,\infty)$, $w=\theta_1-\theta_2$ and suppose that 
    \[
    \partial_t \| w_{<J}\|_{L^2}^2 \leq -2\int ( R^\perp w \cdot \nabla \theta_1 + R^\perp\theta_2  \cdot\nabla w)_{<J} w_{<J}\,dx,
    \]
    where $J\in \Z$ is fixed.
    If, for some time $T>0$, there exists  $C<\I$ such that  
    \[
     \sup_{t\in (0,T)}\frac {\|w_{\geq J}\|_p(t)} {\|w_{<J}\|_p(t)} \leq C,
    \]
    then $w=0$.
\end{theorem}

Above, to be concise we are building the properties of the error into our assumptions. It is reasonable to expect these properties can be deduced directly from properties of Marchand's solutions.

Our second theorem quantifies the extent to which low frequencies are necessarily active in general non-uniqueness scenarios. It is in some sense complimentary to the first.

\begin{theorem}\label{thrm.uniqueness2} Fix $p\in (1,\infty)$. 
    Suppose that $\theta_1,\theta_2$ are distributional solutions to SQG on $\R^2\times (0,T)$ with the same initial data such that, letting $w=\theta_1-\theta_2$, we have  that $w\in L^\infty([0,T);L^p)$ with $\lim_{t\to 0^+}\|w\|_{L^p}(t)=0$, the quantity $\partial_t\|w\|_p^p + p\|\Lambda^{2\alpha /p}w\|^p_p$ is locally bounded on $(0,T)$ and
    \[\partial_t\|w\|_p^p + p\|\Lambda^{2\alpha /p}w\|^p_p \leq - C \int R^\perp w \cdot \nabla \theta_1 \cdot |w|^{p-2}w\,dx.
    \]
    Assume further that  $\|\nabla \theta_1\|_q$ is locally bounded for $t\in (0,T)$ for some $q\in (1,\infty]$.  
    There exists a universal constant $c_*$ so that, defining $J(t)$ via the equation
    \[
    2^{2J(t)\al} = c_* \|\nabla \th_1\|_{q}^{\frac{q\al }{q-1}}(t),
    \]
    if there exists $\de>0$ so that for all $t\in (0,\de)$ we have 
    \[
     \frac { \| w_{\leq J(t)} \|_{p}(t) } {\| w_{> J(t)} \|_{p}(t)} \leq c_*,
    \]
    then $w=0$. 
    \begin{comment}
    If there exists   $J=J(t)$ so that, for all $t\in (0,T)$,
    \[
   \frac{ \|w_{\leq J}\|_{L^p}(t)} {\|w_{>J}\|_{L^p}(t)}\leq \left(C\frac{2^{2J(t)\alpha/p}}{\|\nabla \theta_1\|_{L^q}(t)^{\frac{1}{p(1-\si)}}}-1\right) 
    \]
    where $C$ is a universal constant and $\si = 1+1/\al - 1/(q\al)$, then $w=0$.
    \end{comment}
\end{theorem}

Compared to Theorem \ref{thrm.uniqueness1}, the length scale $2^{-J(t)}$ is dynamic. This leads to a stronger conclusion. It is not clear if Theorem \ref{thrm.uniqueness1} can be improved to match this.

\subsection*{Acknowledgments}

The research of Z. Akridge was supported through a research fellowship awarded by the University of Arkansas Office of Undergraduate Research.

\bigskip \noindent 
The research of Z. Bradshaw was supported in part by the NSF grant DMS-2307097 and the Simons Foundation via a TSM grant (formerly called a collaboration grant).  

\section{Preliminaries}

\subsection{Littlewood-Paley}
We refer the reader to \cite{BCD} for an in-depth treatment of Littlewood-Paley and Besov spaces. Let $\lambda_j=2^j$ be an inverse length and let $B_r$ denote the ball of radius $r$ centered at the origin.  Fix a non-negative, radial cut-off function $\chi\in C_0^\infty(B_{1})$ so that $\chi(\xi)=1$ for all $\xi\in B_{1/2}$. Let $\phi(\xi)=\chi(\lambda_1^{-1}\xi)-\chi(\xi)$ and $\phi_j(\xi)=\phi(\lambda_j^{-1})(\xi)$.  Suppose that $u$ is a vector field of tempered distributions and let $\Delta_j u=\mathcal F^{-1}\phi_j*u$ for $j\geq 0$ and $\Delta_{-1}=\mathcal F^{-1}\chi*u$. Then, $u$ can be written as\[u=\sum_{j\geq -1}\Delta_j u.\]
If $\mathcal F^{-1}\phi_j*u\to 0$ as $j\to -\infty$ in the space of tempered distributions, then we define $\dot \Delta_j u = \mathcal F^{-1}\phi_j*u$ and have
\[u=\sum_{j\in \Z}\dot \Delta_j u.\]
We additionally define
\[
\Delta_{<J} f = \sum_{j<J} \dot \Delta _jf;\quad \Delta_{\geq j} f =f- \Delta_{<J} f,
\]
with the obvious modifications for $\Delta_{\leq J}$ and $\Delta_{>J}$. If we do not specify that $J$ is in integer, then we use $\chi(\la_1^{-1} 2^{J} \xi)$ in the definition of $\Delta_{\leq J}$. We will also write $f_j$ for $\dot \Delta_j f$ when it is convenient.

The Littlewood-Paley formalism is commonly used to define Besov spaces. 
We are primarily interested in Besov spaces with infinite summability index, the norms of which are
\begin{align*}
&||u||_{B^s_{p,\infty}}:= \sup_{-1\leq j<\infty } \lambda_j^s ||\Delta_j u ||_{L^p(\R^n)},
\end{align*}
and
\begin{align*}
&||u||_{\dot B^s_{p,\infty}}:= \sup_{-\infty< j<\infty } \lambda_j^s ||\dot \Delta_j u ||_{L^p(\R^n)}.
\end{align*} 
Importantly for our work, $\dot H^s = \dot B^s_{2,2}$ and $H^s = B^s_{2,2}$.

\subsection{Inequalities}

We make use of the following standard inequalities which are recalled for convenience:
\begin{itemize}
    \item \textbf{Log-convexity of $L^p$ norms.} If $0<p_0<p_1<\infty$ and $f\in L^{p_0}\cap L^{p_1}$, then $f\in L^p$ for all $p\in (p_0,p_1)$ and 
    \[
    \|f\|_{p_\theta} \leq \|f\|^{1-\theta}_{p_0}\|f\|^{\theta}_{p_1}
    \]
    where $\frac{1}{p_\theta}=\frac{1-\theta}{p_0} + \frac{\theta}{p_1}$.
    %\item \textbf{Gr\"onwall's inequality.} 
   % \item \textbf{Hardy-Littlewood-Sobolev inequality.} \crr{ADD ME PLEASE}
    \item \textbf{Bernstein's inequalities.} If  $1\leq p\leq q\leq \infty$ and $\al\in \mathbb N^2$, then
    \[
\| D^\alpha \dot \Delta _jf \|_{L^p} \lesssim_{\al,p}   2^{j|\alpha| } \|\dot \Delta _jf\|_{L^p}; \qquad   \|   \dot \Delta _jf \|_{L^p} \lesssim_{p,q}   2^{j (\frac 2 q - \frac 2 p) } \|\dot \Delta _jf \|_{L^q}.
    \]
    Throughout we only use this in a finite number of settings and use $\la$ to indicate a constant that is valid for this finite selection of Bernstein's inequalities. 
    \item \textbf{Commutator estimate.} The following commutator estimate is due to Miura \cite{Miura}. Let $1\leq s < 2$, $t<1$ with $s+t>1$. Then there exists positive constants $C=C(s,t)$ such that
    \begin{equation}
    \label{ineq.commutator}\|[f,\dot \Delta_j]g\|_{L^2} \leq C2^{-(s+t-1)j}c_j\|f\|_{\dot{H}^s}\|g\|_{\dot{H}^t}
    \end{equation}
    holds for $j\in \Z$, $f\in \dot{H}^s$ and $g\in \dot{H}^s$ with $ \sum_{j\in \Z} c_j^2=1$.
\end{itemize}

\section{Proofs of regularity results}

The next proposition shows that a solution will become small at a certain time if activity in $\th_0$ is concentrated at sufficiently small scales interpreted through the Littlewood-Paley. Theorems \ref{thrm.reg1} and \ref{thrm.reg3} are deduced directly from it while Theorem \ref{thrm.reg2} is deduced from Theorem \ref{thrm.reg1}.

\begin{proposition}\label{PROP}
Let $s\in (2-2\al,2-\al)$. Let $\th_0\in \dot H^s \cap L^2$. Let $\th$ be the strong solution defined on $(0,T(\|\th_0\|_{\dot H^s}))$.
    Fix  $\ga \in (0,1)$. 
    Then, there exist $t \sim (\ga \| \th_0\|_{\dot H^s}^{-1} )^{\frac {2\al} {s-(2-2\al)}}\leq 1$ and $J$ with 
    \[ 
2^{2\al J}\geq C \frac {1} {\ga } \bigg( \frac  {  \| \th_0\|_{\dot H^s}} \ga  \bigg)^{\frac {2\al}{ s-(2-2\al)} }.    
    \]for a universal constant $C$,  so that, if 
    \begin{equation}
        \label{assumption.freqSparse}
         \| \Delta_{<J}\th_0 \|_{\dot H^s} \leq \frac {\ga} {4 C_b} \|\Delta_{\geq J}\th_0\|_{\dot H^s},
    \end{equation}
where $C_b$ is another universal constant, then  
    \[\|\theta(t)\|_{\dot H^{s}}\leq \gamma \|\theta_0\|_{\dot H^{s}}.\]
\end{proposition}

\begin{proof}
    Our argument, which is based on energy estimates, follows the beginning of an estimate of Miura \cite{Miura}. 
    Apply $\dot\Delta_j$ to SQG to obtain
    \begin{equation*}
        \partial_t\theta_j + (-\Delta)^{\alpha}\theta_j = -\dot\Delta_j(u\cdot \nabla \theta).
    \end{equation*}
    Adding $u\cdot \nabla \dot\Delta_j \theta$ we get
    \begin{equation*}
        \partial_t\theta_j + (-\Delta)^{\alpha}\theta_j + u\cdot \nabla   \theta_j = [u,\dot\Delta_j]\nabla \theta,
    \end{equation*}
    where $[u,\dot\Delta_j]$ denotes the commutator.
    Now taking the inner product with $\theta_j$ and using Bernstein's inequality (which involves the constant $\la$ appearing below) we get that
\[
   \frac{1}{2}\frac{d}{dt}\|\theta_j\|_{L^2}^2 + \lambda 2^{2\alpha j}\|\theta_j\|^2_{L^2} + (u\cdot \nabla\Delta_j \theta,\theta_j) = ([u,\Delta_j]\nabla\theta,\theta_j).
\]
From this we infer
\[
 \frac{1}{2}\frac{d}{dt}\|\theta_j\|_{L^2}^2 + \lambda 2^{2\alpha j}\|\theta_j\|^2_{L^2} \leq \|[u,\Delta_j]\nabla\theta\|_{L^2}\|\theta_j\|_{L^2},
\]
and upon dividing through by $\|\theta_j\|_{L^2}$,
\[
  \frac{1}{2}\frac{d}{dt}\|\theta_j\|_{L^2} + \lambda 2^{2\alpha j}\|\theta_j\|_{L^2} \leq \|[u,\Delta_j]\nabla\theta\|_{L^2}.
\]
We now apply the commutator estimate \eqref{ineq.commutator} to obtain  
    \begin{equation*}
    \begin{split}
        \frac{1}{2}\frac{d}{dt}\|\theta_j\|_{L^2} + \lambda 2^{2\alpha j}\|\theta_j\|_{L^2} & \leq \|[u,\Delta_j]\nabla\theta\|_{L^2}\\
        & \leq Cc_j2^{-(2s-2)j}\|u\|_{\dot{H}^{s}}\|\nabla \theta\|_{\dot{H}^{s-1}} \\
        & \leq Cc_j2^{-(2s-2)j}\|\theta\|_{\dot{H}^{s}}^2.
    \end{split}
    \end{equation*}
    Observe that 
    \[
    \frac d {dt} (e^{\la 2^{2\al j}t }\|\theta_j\|_{L^2}(t) ) = e^{\la 2^{2\al j}t } (\frac{d}{dt}\|\theta_j\|_{L^2} + 2\lambda 2^{2\alpha j}\|\theta_j\|_{L^2}  ).
    \]
    Integrating in time   we therefore obtain
    \begin{equation*}
        \|\theta_j\|_{L^2}(t) \leq e^{-2^{2\alpha j}\lambda t}\|\theta_j(0)\|_{L^2} + Cc_j2^{-(2s-2)j}\int_0^te^{-2^{2\alpha j}\lambda (t-s)}\|\theta(s)\|_{\dot H^{s}}^2 \, ds.
    \end{equation*}
    %Where we have switched to the inhomogeneous space. 
    Now multiplying by $2^{sj}$ and  applying the  $l^2(\Z)$-norm with respect to $j$, we get that
    \begin{equation*}
        \|\theta(t)\|_{\dot H^{s}} \leq I + II,
    \end{equation*}
    where
    \begin{equation*}
        I = \left(\sum_{j\in \Z}\left(2^{sj}e^{-2^{2\alpha j}\la t}\|\theta_j(0)\|_{L^2}\right)^2\right)^{1/2},
        \end{equation*}
and
        \begin{equation*}
          II = \left(\sum_{j\in \Z}\left(    Cc_j2^{-sj +2j}\int_0^te^{-2^{2\alpha j}\lambda (t-s)}\|\theta(s)\|_{\dot H^{s}}^2 \, ds\right)^2\right)^{1/2}.
    \end{equation*}
    First we estimate $I$. We will do this for an arbitrary time $t$. We will then give a condition on  $t$ when we bound $II$.  We start by writing 
    \begin{equation*}
       I  \leq I_a+I_b,
    \end{equation*}
    where
    \begin{equation*} 
       I_a  = \left(\sum_{j\geq J}\left(2^{sj}e^{-2^{2\alpha j}\lambda t}\|\theta_j(0)\|_{L^2}\right)^2\right)^{1/2}  \text{ and }
       I_b  = \left(\sum_{j<J}\left(2^{sj}e^{-2^{2\alpha j}\lambda t}\|\theta_j(0)\|_{L^2}\right)^2\right)^{1/2}.
    \end{equation*}
    We evaluate $I_a$ as
    \begin{equation*}
    \begin{split}
        \left(\sum_{j\geq J}\left(2^{sj}e^{-2^{2\alpha j}\lambda t}\|\theta_j(0)\|_{L^2}\right)^2\right)^{1/2} \leq \|\theta_{\geq J}(0)\|_{\dot H^{s}}\left(\sum_{j\geq J}\left(e^{-2^{2\alpha j}\lambda t}\right)^2\right)^{1/2}
 %   
%        \leq \|\theta(0)\|_{H^{2-2\alpha + \epsilon}}\left(\sum_{j\geq J}e^{-2^{2\alpha j+1} \lambda t}\right)^{1/2} \leq \|\theta(0)\|_{H^{2-2\alpha + \epsilon}}\left(\sum_{j\geq J}\frac{1}{2^{2\alpha j+1}t}\right)^{1/2}
\leq \frac{C}{2^{2\alpha J + 1}t}\|\theta_{\geq J}(0)\|_{\dot H^{s}}.
    \end{split}
    \end{equation*}
   We may choose $J=J(t)$ according to $2^{2\alpha J}\geq 2C/(\ga t)$ so that
    \begin{equation}\label{ineq.Ia}
        I_a\leq \frac{\gamma}{4}\|\theta(0)\|_{\dot H^{s}}.
    \end{equation}
    %%Alarm bells, there should be a constant here.
    For the $I_b$, since $e^{-2^{2\alpha j}\lambda t} \leq 1$ for all $j$ and $t$, we simply have 
    \begin{equation}\label{ineq.Ib}
    \begin{split}
        I_b & \leq \left(\sum_{j<J}\left(2^{sj}e^{-2^{2\alpha j}\lambda t}\|\theta_j(0)\|_{L^2}\right)^2\right)^{1/2} \leq C_b\|\Delta_{<J}\theta_0\|_{\dot H^{s}} \leq \frac \ga 4 \|\Delta_{\geq J}\th_0 \|_{\dot H^{s}} ,
        %\\ & \leq C_b\beta\|\theta_0\|_{H^{2-2\alpha + \epsilon}} = \frac{\gamma}{4}\|\theta_0\|_{H^{2-2\alpha + \epsilon}} 
    \end{split}
    \end{equation}
    where we have used the assumption \eqref{assumption.freqSparse}.
    Putting this all together we get that
    \begin{equation} \label{estForI.final}
        I\leq \frac{\gamma}{2} \|\theta_0\|_{\dot H^{s}},
    \end{equation}
    where $t$ is as of yet unspecified and $J=J(t)$.

    For the time-integrated term $II$, we have for $0<s<t$ that
    \[\|\theta(s)\|^2_{H^{s}}\leq 2\|\theta_0\|^2_{H^{s}},%\left [1-C\frac{t\alpha}{\epsilon}\|\Lambda^s\theta_0\|_{L^2}^\frac{\alpha}{\epsilon} \right ]^\frac{-\epsilon}{\alpha}
    \]
    because we are taking $t<T_{min}$ 
    and so  have that
   \[
    II \leq 2\|\theta_0\|_{\dot H^{s}}^2\left(\sum_{j\in \Z}\left(    Cc_j2^{-sj +2j}\int_0^te^{-2^{2\alpha j}\lambda (t-s)} \, ds\right)^2\right)^{1/2}.
    \]
    Evaluating the time integral yields
    \[
    II \leq C\|\theta_0\|_{\dot H^{s}}^2\left(\sum_{j\in \Z}\left(    c_j2^{-sj +2j}\frac 1 {\la 2^{2\al j}} (1-e^{-2^{2\al j}\la t})\right)^2\right)^{1/2}.
    \]
    Observe that by the race-track principle,  $1-e^{-x}\leq x^{(s-(2-2\al))/(2\al)}$ provided $s<2$. Hence
    \begin{align*}
        II &\leq C\|\theta_0\|_{\dot H^{s}}^2\left(\sum_{j\in \Z}\left(    c_j   2^{-sj +2j-2\al j} (2^{2\al j} \la t)^{\frac {s-(2-2\al)} {2\al}}   \right)^2\right)^{1/2}
        \\&\leq C\|\theta_0\|_{\dot H^{s}}^2 t^{(s- (2-2\al))/(2\al)} \| c_j\|_{l^2(\Z)} 
        \\&\leq  C\|\theta_0\|_{\dot H^{s}}^2 t^{(s- (2-2\al))/(2\al)}.
    \end{align*}

To conclude the proof we first choose 
\[
    t= \bigg( \frac \ga {2C \| \th_0\|_{\dot H^s}}   \bigg)^{\frac {2\al}{ s-(2-2\al)}}    ,
\]to ensure 
    \begin{equation}\label{ineq.II}
         II\leq \frac \ga 2 \|\theta_0\|_{\dot H^{s}}.
    \end{equation}
    We then choose $J$   so that \eqref{estForI.final} holds, namely $2^{2\al J}\geq 2C/(\ga t)$. This amounts to having
 \[
2^{2\al J}\geq \frac {2C} {\ga } \bigg( \frac  {2C \| \th_0\|_{\dot H^s}} \ga  \bigg)^{\frac {2\al}{ s-(2-2\al)} }.  
\]

\end{proof}

\bigskip

\begin{proof}[Proof of Theorem \ref{thrm.reg1}]
Note that $T_*\geq T_{min}$ by assumption.  Let $\ga =( T_{min}/T_*)^{(s-2+2\al)/(2\al)}$. If we can show that there exists $t\in (0,T_{min})$ so that 
\[
\| \th (t)\|_{\dot H^s}\leq \bigg( \frac {C_0(s,\al)} {T_*} \bigg)^{\frac {s-2 +2\al}{2\al}},  
\]
then we can solve for a strong solution extending $\th$ on $[t,t+T_*]$ with 
\[
\| \th(s)\|_{\dot H^s}\leq 2 \| \th(t)\|_{\dot H^s}, 
\]
for $s\in (t,t+T_*)$. We have by assumption 
\[
\bigg( \frac {C_0(s,\al)} {T_*} \bigg)^{\frac {s-2 +2\al}{2\al}} = \ga \bigg( \frac {C_0(s,\al)} {T_{min}} \bigg)^{\frac {s-2 +2\al}{2\al}} = \|\th_0\|_{\dot H^s}.
\]
The existence of such a time $t$ follows by Proposition \ref{PROP} under the assumptions of Theorem \ref{thrm.reg1} where we take
\[
2^{2\al J}\geq \frac {2C} {\ga } \bigg( \frac  {2C \| \th_0\|_{\dot H^s}} \ga  \bigg)^{\frac {2\al}{ s-(2-2\al)} }.  
\]     
This holds provided
\[
2^{2\al J} \geq C \|\th_0\|_{\dot H^s}^{1 + \frac {4\al} {s-2+2\al}} T_*^{1+\frac {s-2+s\al}{2\al}},
\]
where the constant  $C$ has replaced those appearing above.  
 
\end{proof}

\begin{proof}[Proof of Theorem \ref{thrm.reg2}]
   This follows by setting $T_*=T_{max}-t$ and applying Theorem \ref{thrm.reg1}. 
\end{proof}

\begin{proof}[Proof of Theorem \ref{thrm.reg3}]
Assume $u\in C((0,T_*);\dot H^s )$ and $T_*$ is a blow-up time. Then, for any $M>0$ there exists $t_M$ so that $\| \th(t_M)\|_{\dot H^s} = M$ and $\| \th(t)\|_{\dot H^s} > M$ for $t\in (t_M,T_*)$. These are called \textit{escape times}.   We will show that smallness of $\| \th \|_{L^\I(0,T_*;\dot B^{2-2\al}_{2,\I})}^2$ implies there are no escape times near $T_*$, indicating $T_*$ is not a blow-up time.

We have for $t$ close enough to $T_*$ that
 \[\big( C_0(s,\al)(T^*-t) \big)^{\frac {s-2+s\al} {2\al}}< \|\theta(t)\|_{H^s},
 \]
or else $T_*$ is not a blow-up time. 

Let $\ga=1/2$. We have for any $t$
\[
\| \Delta_{<J} \th(t)\|_{\dot H^s}^2 \leq \sum_{j<J} 2^{2js} \| \Delta_j \th (t_M)\|_{L^2}^2\lesssim  \| \Delta_{<J}\th \|_{ \dot B^{2-2\al}_{2,\I}}^2 2^{2J(s-2+2\al)}.
\]
Now choose $t'$ so that
\[
t' -t = \bigg(   \frac 1 {4C \|\th(t)\|_{\dot H^s}} \bigg)^{2\al /(s-2+2\al)}.
\]
In the proof of Proposition \ref{PROP} it is clear that the condition on $J$ in Proposition \ref{PROP} holds if
\[
2^{2\al J}\geq \frac {2C}{\ga (t'-t)} = \frac {2C}{\ga  }\big(    {4C \|\th(t)\|_{\dot H^s}} \big)^{2\al /(s-2+2\al)}.
\]
Define $J=J(t)$ so that equality holds in the preceding inequality.
Then 
\[
2^{(2-s+2\al)J} = C(s,\al) \|\th(t)\|_{\dot H^s},
\]
and so
\[
\| \Delta_{<J(t)} \th(t)\|_{\dot H^s}^2  \leq C(s,\al) \| \th \|_{L^\I(0,T_*;\dot B^{2-2\al}_{2,\I})}^2 \|\Delta_{<J(t)}\th(t)\|_{\dot H^s}^2.
\]
Assuming $\| \th \|_{L^\I(0,T_*;\dot B^{2-2\al}_{2,\I})}^2<1$, this implies that
\[
\| \Delta_{<J(t)} \th(t)\|_{\dot H^s}^2  \leq C(s,\al)  \frac {\|\Delta_{<J(t)} \th \|_{L^\I(0,T_*;\dot B^{2-2\al}_{2,\I})}^2} {1-\|\Delta_{<J(t)} \th \|_{L^\I(0,T_*;\dot B^{2-2\al}_{2,\I})}^2} \|\Delta_{\geq J}\th(t)\|_{\dot H^s}^2.
\]
We now require $\|\Delta_{<J(t)} \th(t) \|_{L^\I(0,T_*;\dot B^{2-2\al}_{2,\I})}^2 $ to be small enough that 
\[
C(s,\al)  \frac {\|\Delta_{<J(t)} \th \|_{L^\I(0,T_*;\dot B^{2-2\al}_{2,\I})}^2} {1-\| \Delta_{<J(t)}\th \|_{L^\I(0,T_*;\dot B^{2-2\al}_{2,\I})}^2} \leq \frac 1 {64 C_b^2},
\]
which implies 
  \[
\| \Delta_{<J(t)} \th(t)\|_{\dot H^s}^2  \leq  \frac 1 {\ga^2 16 C_b^2}  \|\Delta_{\geq J(t)}\th(t)\|_{\dot H^s}^2.
\]
Applying Proposition \ref{PROP} we now see that 
\[
\| \th(t')\|_{\dot H^s}\leq \frac 1 2 \|\th(t)\|_{\dot H^s},
\]
indicating $t$ is not an escape time. 
 
\end{proof}

\section{Proofs of uniqueness criteria}

\begin{proof}[Proof of Theorem \ref{thrm.uniqueness1}]

        We start with the following inequality 
   \begin{align*}
    \frac{1}{2}\partial_t\|w_{<J}\|_2^2 + \|\La^{\al} w_{<J}\|_2^2  \,dx&\leq C_1\|w\|_q\| w\|_p\|\nabla w_{<J}\|_p + C_2\|\th_1\|_q\| w\|_p\|\nabla w_{<J}\|_p \\&\leq C2^{J}\|\th_0\|_q\|w\|_p\|  w_{<J}\|_p,
    \end{align*}
    where we require 
    \[
        1 = \frac 2 p + \frac 1 q,
    \]
    with $2<p$ so that $q<\I$ and have used the maximum principle to bound $\th_1$ in terms of $\th_0$.
    From the assumption we have that
    \[
    \|w\|_p\leq \|w_{\geq J}\|_p + \|w_{<J}\|_p \leq (1+C)\|w_{<J}\|_p.
    \]
    Applying this we get
    \[
     \partial_t\|w_{<J}\|_2^2\leq (1+C )2^{J}\|\th_0\|_q\| w_{<J}\|_p^2
    \]
    Then using Bernstein's  inequality we get that
    \[
     \partial_t\|w_{<J}\|_2^2\leq C(1+C)2^{J+ J(2-\frac 4 p)}\|\th_0\|_q\| w_{<J}\|_2^2.
    \]
    Using Gr\"onwall's inequality we see that $w=0$.
\end{proof}

\begin{comment}

ALTERNATIVE PROOF 
\begin{proof}[Proof of Theorem 1.4]
        We start with the following inequality 
    \begin{align*}
    \frac{1}{2}\partial_t\|w_{<J}\|_2^2 + \|\La^{\al} w_{<J}\|_2^2  \,dx
   & \leq C_1\|w\|_q\| w\|_p\|\nabla w_{<J}\|_p + C_2\|\th_1\|_q\| w\|_p\|\nabla w_{<J}\|_p \\ &\leq C2^{J(1-\al)}\|\th_0\|_q\|w\|_p\| \La^{\al}w_{<J}\|_p,
    \end{align*}
    where we require 
    \[
        1 = \frac 2 p + \frac 1 q,
    \]
    with $2<p$ and have used the maximum principle to bound $\th_1$ in terms of $\th_0$.
    From the assumption we have that
    \[
    \|w\|_p\leq \|w_{\geq J}\|_p + \|w_{<J}\|_p \leq (1+C)\|w_{<J}\|_p.
    \]
    Applying this and using Young's multiplicative inequality and Bernstein's inequality we get
    \[
     \partial_t\|w_{<J}\|_2^2\leq C 2^{2J(1-\al)}\|\th_0\|_q^2\| w_{<J}\|_p^2 \leq C 2^{2J(1-\al)} 2^{2(1-2/p)J}\|\th_0\|_q^2\| w_{<J}\|_2^2.
    \]
    Using Gronwall's inequality we see that $w=0$ provided $2^{4J -2\al J -4J/p }  \in L^1(0,T)$.
    Above we required $p>2$. We anticipated getting $2^{2J(t) \al}$ above. We get this if 
    \[
    4J-\al J - 4J/p = 2J \al \iff 4- \frac 4 p = 3\al.
    \]
    If $\al = 1$ we would need $p=4$ (and $q=2$). If $\al = 0$ we would need $p=1$, which is not allowed. Indeed, we need $p>2$ and so $3\al >2$, i.e. $\al>2/3$. 
   \crr{This is sub-optimal because $2^{2J(1-\al)} > 2^J$ when $1>\al$. Hence using diffusion doesn't help.}
    
\end{proof}

\bigskip
\end{comment}

\begin{proof}[Proof of Theorem \ref{thrm.uniqueness2}]
Consider $q<\I$.
Based on our assumptions we have
\[\partial_t\|w\|_p^p + p\|\Lambda^{2\alpha /p}w\|^p_p \leq C \int |R^\perp w \cdot \nabla \th_1 \cdot w^{p-1}| \leq C\|\nabla \th_1\|_q \|w\|_{i} \|w^{p-1}\|_{\frac{i}{p-1}} \leq C\|\nabla \th_1\|_q \|w\|_{i}^{p-1},\]
where $\frac{1}{q}+\frac{p}{i}=1$. Now using the Riesz–Thorin interpolation inequality we get
\[
C\|\nabla w\|_q\|w\|_{i}^p \leq C\|\nabla \th_1 \|_q\|w\|_{p_1}^{p\si}\|w\|_{p_2}^{p(1-\si)},
\]
where we choose $p_2=p$ and $p_1=\frac{p}{1-\alpha}$ and $\si = 1+1/\al - 1/(q\al)$ so that we can use the Hardy-Littlewood-Sobolev lemma to conclude that
\[
C\|\nabla \th_1\|_q\|w\|_{p_1}^{p\si}\|w\|_{p_2}^{p(1-\si)} \leq C \|\nabla \th_1\|_q\|\Lambda^{2\alpha /p}w\|_{p}^{p\si}\|w\|_{p}^{p(1-\si)}.
\]
Now we use Young's product inequality to get
\[
C \|\nabla \th_1\|_q\|\Lambda^{2\alpha /p}w\|_{p}^{p\si}\|w\|_{p}^{p(1-\si)} \leq C\|\nabla \th_1\|_q^{\frac{1}{1-\si}}\|w\|_p^p + \|\Lambda^{2\alpha /p}w\|_{p}^p.
\]
Putting this all together we get that
\[
\partial_t\|w\|_p^p + \|\Lambda^{2\alpha /p}w\|^p_p \leq C\|\nabla \th_1\|_q^{\frac{1}{1-\si}}\|w\|_p^p.
\]
When $q=\I$ arriving at this point is trivial.

By Bernstein's inequality, this  and our assumption imply that
\begin{equation}\label{ineq.aa}
\partial_t\|w\|_p^p+2^{2J\al }\|w_{>J}\|_p^p\leq C\|\nabla \th_1\|_q^{\frac{1}{1-\si}}\|w\|_{p}^{p} \leq  
C\|\nabla \th_1\|_q^{\frac{q\al }{q-1}}\|w_{>J}\|_p^p,
\end{equation}
for $t\in (0,T)$ where we used  our assumption that 
\[
\|w_{\leq J}\|_{p}\leq C \|w_{>J}\|_p
\] 
Now choosing $J(t)$ so that 
\[
 C\|\nabla \th_1\|_q^{\frac{q\al }{q-1}} (t)\leq 2^{2J(t)\al},
\]
implies that 
\[
 \partial_t \|w\|_p^p \leq 0,
\]
and, therefore, $w=0$.
 
\end{proof}

\subsection*{Conflict of interest statement}On behalf of all authors, the corresponding author states that there is no conflict of interest.

	\bibliographystyle{plain}
	\bibliography{main}

\bigskip\noindent 
Zachary Bradshaw, Department of Mathematical Sciences, University of Arkansas, Fayetteville, AR,  USA;
e-mail: \url{zb002@uark.edu}

\medskip\noindent 
Zachary Akridge, Department of Mathematical Sciences, University of Arkansas, Fayetteville, AR,  USA;
 e-mail: \url{zrakridg@uark.edu}

\end{document}